\documentclass[11pt]{article}
\usepackage[margin=1in]{geometry}

\usepackage{graphicx} 
\usepackage{mathtools}

\usepackage{hyperref}
\usepackage{xspace}
\usepackage{amsfonts}
\usepackage{amsmath, amsthm,amssymb}
\usepackage{thm-restate}
\usepackage{thmtools}
\usepackage{etoolbox}
\usepackage[capitalize, noabbrev]{cleveref}
\usepackage{enumitem}
\usepackage[T1]{fontenc}
\usepackage{tikz}
\usetikzlibrary{decorations.pathreplacing, calc}

\newtheorem{theorem}{Theorem}
\newtheorem{corollary}[theorem]{Corollary}
\newtheorem{lemma}[theorem]{Lemma}

\theoremstyle{definition}
\newtheorem{definition}[theorem]{Definition}
\newtheorem{claim}[theorem]{Claim}
\crefname{therm}{Lemma}{Lemmas}
\crefname{claim}{Claim}{Claims}

\newcommand{\Oh}{\mathcal{O}}

\newcommand{\Z}{\mathbb{Z}_{\geq 0}}

\renewcommand{\leq}{\leqslant}
\renewcommand{\geq}{\geqslant}
\renewcommand{\epsilon}{\varepsilon}
\renewcommand{\setminus}{-}

\newcommand{\wei}{\mathfrak{w}}
\newcommand{\esd}{extended strip decomposition\xspace}

\title{Graphs with no long claws: An improved bound for the analog of the Gy\'{a}rf\'{a}s' path argument}
\author{Romain Bourneuf\thanks{Univ. Bordeaux, CNRS, Bordeaux INP, LaBRI, UMR 5800, F-33400, Talence, France, \texttt{romain.bourneuf@ens-lyon.fr}}
\and Jana Masa\v{r}\'{i}kov\'{a}\thanks{University of Warsaw, Poland, \texttt{jnovotna@mimuw.edu.pl}.
Supported by Polish National Science Centre PRELUDIUM 21 grant number 2022/45/N/ST6/04232.}
\and Wojciech Nadara\thanks{University of Warsaw, Poland and Technical University of Denmark, Kongens Lyngby, Denmark \texttt{w.nadara@mimuw.edu.pl}. Supported by Independent Research Fund Denmark grant 2020-2023
(9131-00044B) ``Dynamic Network Analysis'' (while being employed in Denmark) and European Union’s Horizon 2020 research and innovation programme, grant agreement No. 948057 — BOBR (while being employed in Poland).
}
\and Marcin Pilipczuk\thanks{University of Warsaw, Poland, \texttt{m.pilipczuk@mimuw.edu.pl}.
Supported by Polish National Science Centre SONATA BIS-12 grant number 2022/46/E/ST6/00143.}}
\date{}
\begin{document}

\maketitle

\begin{abstract}
    For a fixed integer $t \geq 1$, a \emph{($t$-)long claw}, denoted $S_{t,t,t}$, is
    the unique tree with three leaves, each at distance exactly $t$ from the 
    vertex of degree three. Majewski et al.~[ICALP 2022, ACM ToCT 2024]
    proved an analog of the Gy\'{a}rf\'{a}s' path argument for $S_{t,t,t}$-free graphs:
    given an $n$-vertex $S_{t,t,t}$-free graph, one can delete neighborhoods
    of $\Oh(\log n)$ vertices so that the remainder admits an extended strip decomposition
    (an appropriate generalization of partition into connected components)
    into particles of multiplicatively smaller size.
    This statement has proven to be very useful in designing quasi-polynomial time
    algorithms for \textsc{Maximum Weight Independent Set} and related problems
    in $S_{t,t,t}$-free graphs.
    
    In this work we refine the argument of Majewski et al. and show that 
    a constant number of neighborhoods suffice. 
\end{abstract}

\section{Introduction}\label{sec:intro}

The \textsc{Maximum Weight Independent Set} (\textsc{MWIS}) problem is
a fundamental problem in combinatorial optimization, where the objective
is to find an independent set of vertices in a graph such that the sum
of their weights is maximized. The complexity of \textsc{MWIS}
significantly varies depending on the structure of the input graph.
Some graph classes allow for polynomial-time solutions (e.g., chordal graphs), while in general graphs the problem is NP-hard and hard to approximate within an $n^{1-\varepsilon}$ factor~\cite{Zuckerman07, Hastad99}.
A crucial research direction involves identifying graph classes where
the absence of specific substructures simplifies the \textsc{MWIS}
problem. This leads to the study of hereditary graph classes—those
closed under vertex deletion. Equivalently, they can be defined by a list of forbidden
induced subgraphs. 
In particular, $H$-free graphs -- graphs that do not contain a specific graph $H$ as an induced subgraph -- are of significant interest.

As observed by Alekseev in the 1980s~\cite{alekseev1982effect, Alekseev03}, the problem of \textsc{MWIS} remains NP-hard in $H$-free graphs for most graphs $H$: \textsc{MWIS} remains NP-hard unless  every connected component of $H$ is either a path or a subdivided claw (three paths connected via a vertex).   
Moreover, it cannot even be solved in subexponential time, unless the Exponential-Time Hypothesis fails. Thus, significant attention has been devoted to $P_t$-free graphs (graphs excluding a path on $t$ vertices) and to $S_{t,t,t}$-free graphs (graphs excluding a tree with three leaves connected to the root by paths of length $t$).

The class of $P_4$-free graphs (also known as cographs) is well-structured and many problems, including \textsc{MWIS}, are polynomial-time solvable in it. 
Similarly, \textsc{MWIS} is polynomial-time solvable in $S_{1,1,1}$-free graphs~\cite{MINTY1980284,SBIHI198053} (also known as claw-free graphs).
For the latter, claw-free graphs are closely related to line graphs, and finding a maximum-weight independent set in a line graph corresponds to finding a maximum-weight matching. The next smallest subdivided claw is $S_{2,1,1}$ (also known as the fork), for which a polynomial algorithm was shown by Alekseev~\cite{ALEKSEEV20043} in 2004.

However, moving further is far from being trivial.
Currently, $S_{1,2,2}$ is the smallest open case where the polynomiality of \textsc{MWIS} has not been confirmed.
In 2014, Lokshtanov, Vatshelle, and Villanger \cite{DBLP:conf/soda/LokshantovVV14} showed that \textsc{MWIS} is solvable in polynomial time on $P_5$-free graphs using a framework of potential maximal cliques. The result was extended to $P_6$-free graphs by Grzesik, Klimo\v{s}ov{\'{a}}, Ma.~Pilipczuk, and Mi.~Pilipczuk in 2019~\cite{GrzesikKPP19}.  The case of $P_7$-free graphs remains open. 
Several related partial results have been obtained, see for instance~\cite{AbrishamiCDTTV22,AbrishamiCPRS21}. 
In 2020, in a breakthrough result, Gartland and Lokshtanov~\cite{GartlandL20} obtained a quasipolynomial-time algorithm for \textsc{MWIS} in $P_t$-free graphs. 
The result later got significantly simplified by Pilipczuk, Pilipczuk, and Rz\k{a}\.{z}ewski~\cite{PilipczukPR21} (but still quasipolynomial time) and generalized to a larger class of problems and to $C_{>t}$-free graphs (graphs without induced cycle of length more than $t$)~\cite{GartlandLPPR21}.

In $S_{t,t,t}$-free graphs several partial results were obtained. 
Abrishami, Chudnovsky, Dibek, and Rz\k{a}\.zewski announced a polynomial-time algorithm for \textsc{MWIS} in $S_{t,t,t}$-free graphs of bounded degree~\cite{ACDR21} that got later improved~\cite{AbrishamiCPR24} where the same set of authors and Pilipczuk also provide a polynomial-time algorithm for \textsc{MWIS} in graphs excluding a fixed graph whose every component is a subdivided claw as an induced subgraph, and a fixed biclique as a subgraph.
However, much less was known for the general cases (even for a specific small $t$) up until recently. 
In 2023, another breakthrough, a quasipolynomial algorithm for $S_{t,t,t}$-free graphs, was presented by Gartland, Lokshtanov, Masa\v{r}{\'{\i}}k, Pilipczuk, Pilipczuk, and Rz\k{a}\.zewski \cite{Sttt-qpoly}. 
This progress corroborates the conjecture that \textsc{MWIS} is polynomial-time solvable in $H$-free graphs for all the open cases left by the original hardness reductions, that is, whenever $H$ is a forest whose every connected component has at most three leaves. 

The work~\cite{Sttt-qpoly} is built upon two pillars. The first one 
is the branching strategy similar to the first $P_t$-free paper \cite{GartlandL20}.
The second pillar is a work~\cite{Sttt_logn23} that
gives a structural result for $S_{t,t,t}$-free graphs that is an analog 
of the Gy\'{a}rf\'{a}s' path argument from $P_t$-free graphs. 
In this work, we improve the main result of \cite{Sttt_logn23}
to a form with near optimal parameter dependency.
Before we state the result formally, let us discuss the relevant results and techniques.

Firstly, we return to $P_t$-free graphs. 
In the 1980s, Gy\'{a}rf\'{a}s showed that for every fixed $t$ the class of $P_t$-free graphs is $\chi$-bounded~\cite{gyarfas1,gyarfas2}.
Bacs{\'{o}}, Lokshtanov, Marx, Pilipczuk, Tuza,  and van Leeuwen ~\cite{DBLP:journals/algorithmica/BacsoLMPTL19} gave a subexponential-time algorithm for \textsc{MWIS} in $P_t$-free graphs, 
using the following important corollary of the Gy\'{a}rf\'{a}s' path argument. 
\begin{theorem}\label{thm:gyarfas}
Given an $n$-vertex graph $G$ with nonnegative vertex weights, one can in polynomial time find an induced path $Q$ in $G$
such that every connected component of $G-N[V(Q)]$ has weight at most half of the total weight of $V(G)$.
\end{theorem}
Observe that in $P_t$-free graphs, $Q$ contains at most $t-1$ vertices. In particular, having connected components of multiplicatively smaller size after removing the neighborhood of a constant number of vertices became very useful in the recursive approach.

For the progress in general $S_{t,t,t}$-free graphs, the paper~\cite{Sttt_logn23} shows that the notion of an \emph{extended strip decomposition}, developed by Chudnovsky and Seymour in their project to understand claw-free graphs~\cite{claw-free-survey}, became useful as an analog of the Gy\'{a}rf\'{a}s' path argument.
For a formal definition of extended strip decomposition, we refer to \cref{sec:prelim}. 
For an intuition to develop, we remark that in an extended strip decomposition of a graph,  we can distinguish \emph{particles} (specific induced subgraphs of the graph). 
The crucial property is that a maximum weight independent set of the whole graph can be combined from solutions obtained independently in individual particles~(\cite{ChudnovskyPPT}). 
Thus, `particles' of an extended strip decomposition can be viewed as analogs of connected components. Particles of multiplicatively smaller size in $S_{t,t,t}$-free graphs have been obtained in \cite{Sttt_logn23}. 

\begin{theorem}[\cite{Sttt_logn23}]\label{thm:logn-result}
  Given an $n$-vertex graph $G$ with nonnegative vertex weights and an integer $t \geq 1$, one can in polynomial time either:
\begin{itemize}
\item output an induced copy of $S_{t,t,t}$ in $G$, or
\item output a set $\mathcal{P}$ consisting of at most $11 \log n+6$ induced paths in $G$,
  each of length at most $t+1$, and a rigid \emph{extended strip decomposition} of
  $G - N[\bigcup_{P \in \mathcal{P}} V(P)]$ whose every particle has weight at most half of the total weight of $V(G)$. 
\end{itemize}
\end{theorem}

Using this structural result, the authors provided a subexponential-time algorithm with running time $2^{\Oh(\sqrt{n}\log n)}$ and a quasipolynomial-time approximation scheme with running time 
$2^{\Oh(\varepsilon^{-1} \log^{5} n)}$.
Both results improved and simplified the state-of-the-art knowledge by Chudnovsky, Pilipczuk, Pilipczuk, and Thomass{\'{e}}~\cite{ChudnovskyPPT} from 2020.

This result has already found several applications besides simplifying the analysis of the quasipoly\-nomial-time algorithm~\cite{AbrishamiCPR24}.
Most importantly, it became as a black box one of the pillars of the quasipolynomial-time algorithm for \textsc{MWIS} in $S_{t,t,t}$-free graphs~\cite{Sttt-qpoly}. 
There, the key structural ingredient for the algorithm is the following. 
For a fixed integer $t\ge 1$, given a graph $G$, one can in polynomial time find either an induced $S_{t,t,t}$ in $G$, or a balanced separator consisting of $\Oh(\log |V(G)|)$ vertex neighborhoods in $G$, or an extended strip decomposition of $G$ with each particle of weight multiplicatively smaller than the weight of $G$. We remark that the third output gives an extended strip decomposition of the whole graph, in contrast with~\cref{thm:logn-result}.

\textbf{Our Contribution} We address an open question from~\cite{Sttt_logn23} that asked to eliminate the logarithmic factor in output 2 of \cref{thm:logn-result}.
Indeed, it turns out that removing neighborhoods of a constant number of \emph{short} paths is sufficient to obtain an extended strip decomposition with each particle of weight multiplicatively smaller than the total weight of the graph. 
Our main contribution is the following theorem.

\begin{restatable}{theorem}{ourmaintheorem}\label{thm:main-thm}
Given a graph $G = (V, E)$ with nonnegative vertex weights $\wei: V \rightarrow \mathbb{N}$ and an integer $t \geq 1$, one can in polynomial time either:
\begin{itemize}
    \item output an induced copy of $S_{t,t,t}$ in $G$, or
    \item output a set $\mathcal{S}$ of size at most $3t+11$, and a rigid extended strip decomposition of $G - N[\mathcal{S}]$ whose every particle has weight at most $\wei(V)/2$. 
\end{itemize}
\end{restatable}

Similarly, as in~\cite{Sttt_logn23}, our starting point is the Gy\'{a}rf\'{a}s' path, refining the selection of splitting vertices. If its length is within a small multiple of $t$, we even have a stronger result of components of half of the weight of $G$. 
Otherwise, we identify specific vertices of the Gy\'{a}rf\'{a}s' path and query the three-in-a-tree theorem, originally developed by Chudnovsky and Seymour \cite{DBLP:journals/combinatorica/ChudnovskyS10} with improved running time by Lai, Lu, and Thorup~\cite{LaiLT20}, on them.
The output of the three-in-a-tree theorem is either an induced tree connecting the given vertices or an extended strip decomposition of $G$. 
By the choice of the vertices we make sure the second output applies, otherwise an induced copy of $S_{t,t,t}$ is found.
The main difference from~\cite{Sttt_logn23} lies in a more involved choice of the vertices with a better understanding of the structure of the extended strip decomposition. 
Among others, we use the obvious but powerful property of the last two vertices of the Gy\'{a}rf\'{a}s' path. 
After removing the neighborhood of the Gy\'{a}rf\'{a}s' path up to its second-to-last vertex, the largest connected component is still `big' (has weight larger than half of the weight of $G$). However, after removing the neighborhood of the last vertex, its weight suddenly drops below this threshold.
This observation lets us look more closely at the interaction of the `big' particle and the `big' connected component.

We remark that our improvement does not yield a polynomial algorithm. 
Applying it to \cite{Sttt-qpoly} eliminates one logarithmic factor in the exponent. 
However, it significantly simplifies the algorithmic results for \textsc{MWIS}
of~\cite{AbrishamiCPR24}: the polynomial-time algorithm in 
$S_{t,t,t}$-free graphs of bounded degree is now immediate (just apply
Theorem~\ref{thm:main-thm}, exhaustively branch on $N[\mathcal{S}]$, and recurse 
on the particles of the extended strip decomposition)
and the algorithm for $S_{t,t,t}$-free $K_{s,s}$-free graphs can be significantly simplified.

\medskip

The majority of this work was done during the
Sparse Graphs Coalition workshop 
``Algebraic, extremal, and structural methods and problems in graph colouring''
that happened online in February 2024. 

We remark that a very similar result to Theorem~\ref{thm:main-thm}
has been independently obtained
by Chudnovsky, Codsi, Milani\v{c}, Lokshtanov, and Sivashankar~\cite{CCMLS25},
via a slightly different proof.

\section{Preliminaries}\label{sec:prelim}

\paragraph*{Notation.} 

All graphs are simple. Let $G$ be a graph. For $X \subseteq V(G)$, by $G[X]$ we denote the subgraph of $G$ induced by $X$, i.e., $(X, \{uv \in E(G) : u,v \in X\})$.
If the graph $G$ is clear from the context, we will often identify induced subgraphs with their vertex sets.

For a vertex $v$, by $N_G(v)$ we denote the set of neighbors of $v$, and by $N_G[v]$ we denote the set $N_G(v) \cup \{v\}$.
For a set $X \subseteq V(G)$, we also define $N_G(X) \coloneqq  \bigcup_{v \in X} N_G(v) \setminus X$, and $N_G[X] = N_G(X) \cup X$.
If it does not lead to confusion, we omit the subscript and write simply $N( \cdot)$ and $N[ \cdot]$.
By $\binom{V(G)}{3}$ we mean the set of all subsets of $V(G)$ of size 3.
By $T(G)$, we denote the set of all triangles in $G$. Note that $T(G) \subseteq \binom{V(G)}{3}$.
Similarly to writing $xy \in E(G)$, we will write $xyz \in T(G)$ to indicate that $G[\{x,y,z\}]$ is a triangle.
We say that two sets $X,Y \subseteq V(G)$ \emph{touch} if $X \cap Y \neq \emptyset$ or there is an edge with one end in $X$ and another in $Y$.
For a family $\mathcal{Q}$ of sets, by $\bigcup \mathcal{Q}$ we denote $\bigcup_{Q \in \mathcal{Q}} Q$.
For a function $\wei : V \to \Z$ and subset $V' \subseteq V$, we denote $\wei(V')\coloneqq \sum_{v \in V'} \wei(v)$.

A graph $G$ is said to be \emph{F-free} for some graph $F$ if no induced subgraph of $G$ is isomorphic to $F$. 
For integers $t,a,b,c > 0$, by $P_t$ we denote the path on $t$
vertices, and by $S_{a,b,c}$, we denote the tree with three leaves at respective distance $a$, $b$, and $c$ from the unique vertex of degree $3$ of the tree.

\paragraph*{Extended strip decompositions.} 
\begin{definition}[Extended strip decomposition, \cite{DBLP:journals/combinatorica/ChudnovskyS10}]
An \emph{extended strip decomposition} of a graph $G$ is a pair $(H, \eta)$ that consists of:
\begin{itemize}
\item a simple graph $H$,
\item a set $\eta(x) \subseteq V(G)$ for every $x \in V(H)$,
\item a set $\eta(xy) \subseteq V(G)$ for every $xy \in E(H)$, and its subsets $\eta(xy,x),\eta(xy,y) \subseteq \eta(xy)$,
\item a set $\eta(xyz) \subseteq V(G)$  for every $xyz \in T(H)$,
\end{itemize}
which satisfy the following properties:
\begin{enumerate}
\item $\{\eta(o)~|~o \in V(H)\cup E(H) \cup T(H)\}$ is a partition of $V(G)$,
\item for every $x \in V(H)$ and every distinct $y,z \in N_H(x)$, the set $\eta(xy,x)$ is complete to $\eta(xz,x)$,
\item every $uv \in E(G)$ is contained in one of the sets $\eta(o)$ for $o \in V(H) \cup E(H)\cup T(H)$, or is as follows:
\begin{itemize}
\item $u \in \eta(xy,x), v\in \eta(xz,x)$ for some $x \in V(H)$ and $y,z \in N_H(x)$, or
\item $u \in \eta(xy,x), v\in \eta(x)$ for some $xy \in E(H)$, or
\item $u \in \eta(xyz)$ and $v\in \eta(xy,x) \cap \eta(xy,y)$ for some $xyz \in T(H)$. 
\end{itemize}
\end{enumerate}
\end{definition}
See~\cite[Figure 1]{Sttt_logn23} for an illustration of the definition.

For an edge $xy \in E(H)$, the sets $\eta(xy, x)$ and $\eta(xy, y)$ are called the \emph{interfaces} of the edge $xy$.
An extended strip decomposition $(H,\eta)$ is \emph{rigid} if (i) for every $xy \in E(H)$ it holds that $\eta(xy,x) \neq \emptyset$,
and (ii) for every $x \in V(H)$ such that $x$ is an isolated vertex it holds that $\eta(x) \neq \emptyset$.
Observe that if we \emph{restrict} $\eta$ to $V'\subset V(G)$, i.e.\ we keep in $\eta$ only vertices of $V'$, $(H,\eta)$ after the restriction remains an extended strip decomposition, but it might not be rigid anymore.

We say that a vertex $v \in V(G)$ is \emph{peripheral} in $(H,\eta)$ if there is a degree-one vertex $x$ of $H$,
such that $\eta(xy,x)=\{v\}$, where $y$ is the (unique) neighbor of $x$ in $H$.
Note that we can also assume $\eta(x) = \emptyset$, as otherwise all vertices
from $\eta(x)$ can be moved to $\eta(xy)$. 
For a set $Z \subseteq V(G)$, we say that $(H, \eta)$ is an \emph{extended strip decomposition of $(G,Z)$} if $H$ has $|Z|$ degree-one vertices and each vertex of $Z$ is peripheral in $(H,\eta)$. 
The following theorem by Chudnovsky and Seymour~\cite{DBLP:journals/combinatorica/ChudnovskyS10}
is a slight strengthening of their celebrated solution of the famous \emph{three-in-a-tree} problem.

\begin{theorem}[{Chudnovsky, Seymour~\cite[Section 6]{DBLP:journals/combinatorica/ChudnovskyS10}}]\label{thm:three-in-a-tree}
Let $G$ be an $n$-vertex  graph and $Z \subseteq V(G)$ with $|Z| \geq 2$.
There is an algorithm that runs in time $\Oh(n^4)$ and returns one of the following:
\begin{itemize}
\item an induced subtree of $G$ containing at least three elements of $Z$,
\item a rigid extended strip decomposition of $(G,Z)$.
\end{itemize}
\end{theorem}

The running time was improved to $\Oh(m\log^2n)$ (where $m=|E|$) in~\cite{LaiLT20}.

\begin{definition}
Let $(H, \eta)$ be an extended strip decomposition of a graph $G$.
We distinguish the following special subsets of $V(G)$:
\begin{align*}
\textrm{vertex particle:} &\quad A_{x} \coloneqq  \eta(x) \text{ for each } x \in V(H),\\
\textrm{edge interior particle:} &\quad A_{xy}^{\perp} \coloneqq  \eta(xy) \setminus (\eta(xy,x) \cup \eta(xy,y)) \text{ for each } xy \in E(H),\\
\textrm{half-edge particle:} &\quad A_{xy}^{x} \coloneqq   \eta(x) \cup \eta(xy) \setminus \eta(xy,y) \text{ for each } xy \in E(H),\\
\textrm{full edge particle:} &\quad A_{xy}^{xy} \coloneqq  \eta(x) \cup \eta(y) \cup \eta(xy) \cup \bigcup_{z ~:~ xyz \in T(H)} \eta(xyz) \text{ for each } xy \in E(H),\\
\textrm{triangle particle:} &\quad A_{xyz} \coloneqq  \eta(xyz) \text{ for each } xyz \in T(H).
\end{align*}
We refer to these sets as \emph{particles} and discern between their respective \emph{types}.
\end{definition}

A vertex particle $A_x$ is \emph{trivial} if $x$ is an isolated vertex in $H$.
Similarly, an extended strip decomposition $(H,\eta)$ is \emph{trivial} if $H$ is an edgeless graph. 
Given a weight function $\wei : V(G) \to \mathbb{N}$, we say that a particle is \emph{small} if its weight is at most $\frac{\wei(V(G))}{2}$, and an \esd $(H,\eta)$ of $G$ is \emph{refined} if all its particles are small. 

\section{The result}

We restate the main theorem here.
\ourmaintheorem*

\subsection{Overview of the Proof}

Before the actual proof, we show several properties of extended strip decompositions.
For the purposes of this overview, the most important is \cref{lem:gyarfas-path-inside-edges}, which states that if $(H, \eta)$ is an \esd of a graph $G$ then any induced path $P$ (in $G$) between peripheral vertices is entirely contained in edge particles. Moreover, in each edge particle, $P$ has at most one vertex in each interface.

Let us now describe the key ideas of the proof of \cref{thm:main-thm}.
We start by considering a Gy\'{a}rf\'{a}s' path $Q$ in $G$. 
If $Q$ is short, we can set $\mathcal{S} = V(Q)$ and return a trivial \esd of $G - N[\mathcal{S}]$ satisfying the requirements of~\cref{thm:main-thm}, so assume that $Q$ is long. 
First, we mark specific vertices of $Q$ on which we will query the three-in-a-tree theorem: The first vertex of $Q$ (call it $x$) and the vertices at distance $t+4$ (call it $y$) and $t+2$ (call it $z$) from the end of $Q$.
Before querying the three-in-a-tree theorem, we remove the neighborhood outside $Q$ of three subpaths of $Q$: the subpaths of $Q$ starting at $x$ and $z$ and taking the next $t$ vertices in $Q$, respectively, and the subpath taking the $t$ vertices of $Q$ preceding $y$. We also remove the vertex 
of $Q$ between $y$ and $z$.  
Observe that this selection ensures that the three-in-a-tree theorem either outputs an induced copy of $S_{t,t,t}$ (which is a desired output of ~\cref{thm:main-thm}) or an extended strip decomposition of the graph minus the mentioned neighborhoods.

We take a closer look at the returned extended strip decomposition. If all its particles are small, we can simply return it, so assume it has a big particle $A$. 
Let us assume that $A$ is a full-edge particle, i.e., $A=A_{pq}^{pq}$, the other cases being easy to handle.
We distinguish two cases depending on whether the subpath $Q_1$ of $Q$ from $x$ to $y$ intersects $A$ or not.
If $Q_1$ is disjoint from $A$, the only vertices of $Q_1$ that can have a neighbor in $A$ are vertices in an interface of a neighboring edge of $pq$ in $H$. 
\cref{lem:gyarfas-path-inside-edges} implies that there can be at most four of them. 
Finally, using that $Q$ is a Gy\'{a}rf\'{a}s' path, we observe that after removing the neighborhoods of the three previously mentioned subpaths, of these four vertices and of four other well-chosen vertices, every remaining connected component has weight at most half of the total weight of $G$.
If $Q_1$ intersects $A$, denote by $Q_2$ be the subpath of $Q$ from $z$ to the end of $Q$. The fact that $Q$ is induced and \cref{lem:gyarfas-path-inside-edges} together imply that $Q_2$ does not have any neighbor in $A$. From the property of  Gy\'{a}rf\'{a}s' path, letting $\ell$ denote the last vertex of $Q$, we have that $G\setminus N[V(Q)-\ell]$ has a big connected component $C$ but $G\setminus N[V(Q)]$ doesn't. Therefore, $\ell$ has a neighbor in $V(C)$. Two big parts must intersect, so $V(C)$ and $A$ intersect. It then follows from \cref{lem:gyarfas-path-inside-edges} that $V(C) \subseteq A$. This contradicts the fact that $Q_2$ does not have any neighbor in $A$.

\subsection{Proof of~\cref{thm:main-thm}}
We start with several properties of extended strip decompositions.

\begin{lemma} \label{lem:small-dom-ngbhd}
    Let $(H, \eta)$ be a rigid extended strip decomposition of a graph $G$. Let $A = A_{xy}^{xy}$ be a full-edge particle. There exists a set $X_A$ of size at most 2 such that $N_G(A) \subseteq N_G(X_A)$.
\end{lemma}
\begin{proof}
    Since $(H, \eta)$ is rigid, there exist $v_x \in \eta(xy, x)$ and $v_y \in \eta(xy, y)$.
    Set $X_A = \{v_x, v_y\}$.
    Let $u \in N_G(A)$. Since $u \notin A$ is adjacent to a vertex in $A = A^{xy}_{xy}$, there exists $z \in \{x, y\}$ and an edge $f \neq xy$ incident to $z$ such that $u \in \eta(f, z)$. 
    Then, $u$ is adjacent to $v_z \in \{v_x, v_y\}$ so $u \in N_G(X_A)$.
\end{proof}
We note that the above observation was mentioned in~\cite[Observation 8]{Sttt_logn23}, we include it here for completeness.

As discussed above, if $(H, \eta)$ is an extended strip decomposition of a graph $G$, if $G'$ is an induced subgraph of $G$ and $\eta'$ is the restriction of $\eta$ to $V(G')$ then $(H, \eta')$ is an extended strip decomposition of $G'$, but it is not necessarily rigid anymore. 
The following lemma shows how to modify it to obtain a rigid extended strip decomposition, without increasing too much the maximum weight of a particle.

\begin{lemma}\label{lem:make-rigid}
    Let $G = (V, E)$ be a graph with nonnegative vertex weights $\wei: V \rightarrow \mathbb{N}$, $w \geq \wei(V)/2$ and $(H, \eta)$ be an extended strip decomposition of $G$ whose every particle has weight at most $w$.
    Then, one can in polynomial time compute either a rigid extended strip decomposition
    of $G$ whose every particle has weight at most $w$, or a set $X \subseteq V(G)$ of size at most $2$ such that every connected component of $G-N[X]$ has weight at most $w$. 
\end{lemma}

\begin{proof}
For the purposes of the proof, we introduce the notion of \emph{relaxed extended strip decomposition}, which is similar to that of extended strip decomposition, except that we have a set $\eta(xyz)$ for every triple of vertices $xyz$ in $H$ and not only for every triangle $xyz$ in $H$. 
We require the same properties as in an extended strip decomposition when replacing $T(H)$ by $\binom{V(H)}{3}$.
The particles of a relaxed extended strip decomposition are the vertex particles, the full edge particles and the triple particles, each defined similarly to the corresponding particle in an extended strip decomposition when replacing $T(H)$ by $\binom{V(H)}{3}$.
Note that every extended strip decomposition naturally corresponds to a relaxed extended strip decomposition and that a relaxed extended strip decomposition such that $\eta(xyz) = \emptyset$ whenever $xyz \notin T(H)$ naturally yields an extended strip decomposition.

We view $(H, \eta)$ as a relaxed extended strip decomposition of $G$, and we start by ensuring that every edge $xy \in E(H)$ satisfies $\eta(xy, x) \neq \emptyset$. 
Let $xy \in E(H)$ be such that $\eta(xy, x) = \emptyset$.
If $\eta(xy, y) = \emptyset$ too, let $H'$ be the graph obtained from $H$ by removing the edge $xy$ and adding an isolated vertex $v$. For every $o \in V(H) \cup E(H) \cup \binom{V(H)}{3} \setminus \{xy\}$, set $\eta'(o) = \eta(o)$, and set $\eta'(v) = \eta(xy)$ and $\eta'(vu_1u_2) = \emptyset$ for every newly created triple of vertices. 
If $\eta(xy, y) \neq \emptyset$, let $H'$ be the graph obtained from $H$ by removing the edge $xy$ and adding a vertex $z$ adjacent only to $y$. For every $o \in V(H) \cup E(H) \cup \binom{V(H)}{3} \setminus \{xy\}$, set $\eta'(o) = \eta(o)$, and set $\eta'(zy) = \eta'(zy, z) = \eta(xy)$, $\eta'(zy, y) = \eta(xy, y)$, $\eta'(z) = \emptyset$ and $\eta'(zu_1u_2) = \emptyset$ for every newly created triple of vertices.

Note that in both cases $(H', \eta')$ is still a relaxed extended strip decomposition of $G$ whose every particle has weight at most $w$.
When going from $(H, \eta)$ to $(H', \eta')$, the number of edges with an empty interface decreases strictly.
Therefore, by repeating this operation, we eventually get that every edge $xy \in E(H)$ satisfies $\eta(xy, x) \neq \emptyset$.

Suppose now that there exists an isolated vertex $x$ of $H$ and a triple $xyz \in \binom{V(H)}{3}$ such that $\eta(xyz) \neq \emptyset$.
If $yz \in E(H)$, let $H' = H$ and for every $o \in V(H) \cup E(H) \cup \binom{V(H)}{3} \setminus \{xyz, yz\}$, set $\eta'(o) = \eta(o)$, set $\eta'(xyz) = \emptyset$ and $\eta'(yz) = \eta(yz) \cup \eta(xyz), \eta'(yz, y) = \eta(yz, y), \eta'(yz, z) = \eta(yz, z)$.
If $yz \notin E(H)$, let $H'$ be the graph obtained from $H$ by adding an isolated vertex $v$. For every $o \in V(H) \cup E(H) \cup \binom{V(H)}{3} \setminus \{xyz\}$, set $\eta'(o) = \eta(o)$, and set $\eta'(v) = \eta(xyz)$, $\eta'(xyz) = \emptyset$ and $\eta'(vu_1u_2) = \emptyset$ for every newly created triple of vertices. 

Observe that in both cases $(H', \eta')$ is still a relaxed extended strip decomposition of $G$ whose every particle has weight at most $w$.
When going from $(H, \eta)$ to $(H', \eta')$, we create no edge with an empty interface and the number of triples $xyz \in \binom{V(H)}{3}$ such that $\eta(xyz) \neq \emptyset$ and $x$ is isolated in $H$ decreases strictly.
Therefore, by repeating this operation, we eventually get that no edge has an empty interface and for every triple $xyz \in \binom{V(H)}{3}$ such that $x$ is isolated, we have $\eta(xyz) = \emptyset$.

Therefore, if $x$ is an isolated vertex of $H$ such that $\eta(x) = \emptyset$ then after removing $x$ from $H$, what we obtain is again an extended strip decomposition of $G$ with the same properties.
By iteratively removing such vertices, we eventually get a rigid relaxed extended strip decomposition whose every particle has weight at most $w$.

We now show how to convert this rigid relaxed extended strip decomposition $(H, \eta)$ into a rigid extended strip decomposition whose every particle has weight at most $w$.
If there is no $xyz \in \binom{V(H)}{3}$ such that $\eta(xyz) \neq \emptyset$ and $xyz \notin T(H)$ then restricting $\eta$ to $V(H) \cup E(H) \cup T(H)$ gives the desired rigid extended strip decomposition. Suppose now that there exists a triple $xyz \in \binom{V(H)}{3}$ such that $\eta(xyz) \neq \emptyset$ and $xyz \notin T(H)$. We consider several cases. \begin{itemize}
    \item If $\eta(xyz)$ is not adjacent to $V(G) - \eta(xyz)$, let $H'$ be the graph obtained from $H$ by adding an isolated vertex $v$. For every $o \in V(H) \cup E(H) \cup \binom{V(H)}{3} \setminus \{xyz\}$, set $\eta'(o) = \eta(o)$, and set $\eta'(v) = \eta(xyz)$, $\eta'(xyz) = \emptyset$ and $\eta'(vu_1u_2) = \emptyset$ for every newly created triple of vertices. 
    \item If $\eta(xyz)$ is adjacent to only one of $\{\eta(xy), \eta(xz), \eta(yz)\}$, say $\eta(xy)$, let $H' = H$ and for every $o \in V(H) \cup E(H) \cup \binom{V(H)}{3} \setminus \{xyz, xy\}$, set $\eta'(o) = \eta(o)$, and set $\eta'(xy) = \eta(xy) \cup \eta(xyz)$ and $\eta'(xyz) = \emptyset$. Importantly, we set $\eta'(xy, x) = \eta(xy, x)$ and $\eta'(xy, y) = \eta(xy, y)$. 
    \item Otherwise, since $xyz \notin T(H)$, only two of $\{xy, xz, yz\}$ are edges of $H$, say $xy$ and $xz$. Then, $\eta(xyz)$ is adjacent to both $\eta(xy)$ and $\eta(xz)$. Again, we set $H' = H$, for every $o \in V(H) \cup E(H) \cup \binom{V(H)}{3} \setminus \{xyz, x\}$, set $\eta'(o) = \eta(o)$, and set $\eta'(x) = \eta(x) \cup \eta(xyz)$ and $\eta'(xyz) = \emptyset$.
\end{itemize}

In each case, $(H', \eta')$ is a rigid relaxed extended strip decomposition of $G$ and the number of triples $xyz \in \binom{V(H)}{3}$ such that $\eta(xyz) \neq \emptyset$ and $xyz \notin T(H)$ decreases strictly. Thus, by iterating this operation we eventually obtain a rigid relaxed extended strip decomposition of $G$ with no $xyz \in \binom{V(H)}{3}$ such that $\eta(xyz) \neq \emptyset$ and $xyz \notin T(H)$. If this relaxed extended strip decomposition has no particle of weight strictly greater than $w$ then as before we can turn it into the desired rigid extended strip decomposition of $G$.

Therefore, we can assume that one of the above operations created a particle of weight strictly greater than $w$.
Note that in the first case, we do not increase the weight of any pre-existing particle, and the only non-empty particle we create is $A_v = \eta'(v) = \eta(xyz)$ which has weight at most $w$ by assumption.
Similarly, in the second case, we do not increase the weight of any particle.
Therefore, we created a particle of weight strictly greater than $w$ going from some rigid relaxed extended strip decomposition $(H, \eta)$ to another rigid relaxed extended strip decomposition $(H', \eta')$ using the third case.
In this case, all particles keep the same weight except for $A_{xyz}$ which becomes empty, $A_x$, and the full edge particles $A_{xa}^{xa}$ ($a \notin \{y, z\}$). 
Since after the modification we have $A_x \subseteq A_{xy}^{xy}$ and the weight of $A_{xy}^{xy}$ did not change then after the modification we still have $\wei(A_x) \leq w$.
Thus, there exists a full-edge particle $A' = A_{xa}^{xa}$ of $(H', \eta')$ such that $\wei(A') > w$, and $a \notin \{y, z\}$.

Since $(H, \eta)$ is rigid, as in~\cref{lem:small-dom-ngbhd} there exists a set $X_{A'}$ of size at most 2 such that $N_G(A') \subseteq N_G(X_{A'})$. Therefore, in $G - N_G[X_{A'}]$, the remaining vertices of $A'$ are disconnected from $V - A'$.
Furthermore, $V - A'$ has weight at most $\wei(V) - \wei(A') \leq w$.
Let $A = A_{xa}^{xa}$ be the corresponding particle of $(H, \eta)$. 
Then, $A' = A \cup \eta(xyz)$ but since $a \notin \{y, z\}$ then $\eta(xyz)$ is not adjacent to $A$ in $G$. 
Hence, for every connected component $C$ of $G - N_G[X_{A'}]$, we have either $V(C) \subseteq V - A'$ or $V(C) \subseteq \eta(xyz)$ or $V(C) \subseteq A$. In all three cases we have $\wei(V(C)) \leq w$ since $(H, \eta)$ is refined.

Finally, note that this proof is algorithmic and can be implemented in polynomial time (in the sizes of $H$ and $G$). Indeed, every modification we perform on the relaxed extended strip decomposition can be done in polynomial time, and we can find the next one in polynomial time.
Therefore, we just need to argue that the size of the relaxed extended strip decomposition remains within a polynomial factor of the size of the initial \esd.
First, we perform a polynomial number of modifications to ensure that no edge has an empty interface, and each such modification adds at most one vertex. Then, to ensure that there is no triple $xyz$ such that $x$ is isolated and $\eta(xyz) \neq \emptyset$ we again perform a polynomial number of modifications, each adding at most one vertex. Then, we remove isolated vertices, which cannot increase the size of the relaxed \esd. Finally, we perform a polynomial (in the size of the current relaxed \esd) number of operations to turn the relaxed \esd into a proper \esd.
\end{proof}

We now look at induced paths in $G$ whose endpoints are peripheral in $(H, \eta)$, with particular emphasis on where their vertices appear in $(H,\eta)$. 
The following lemmata characterize how an induced path (in $G$) between peripheral vertices can traverse an extended strip decomposition. 

\begin{lemma} \label{lem:enter-particle}
    Let $(H, \eta)$ be an extended strip decomposition of a graph $G$ and $z \in V(G)$ be a peripheral vertex in $(H, \eta)$. Let $Q = (z = x_0, x_1, \ldots, x_{k-1}, x_k)$ be an induced path in $G$ with endpoint $z$.
    Let $A = A^{pq}_{pq}$ be a full edge particle such that $z \notin \eta(pq)$. If $V(Q) \cap N[A] \neq \emptyset$, there exists an edge $f \neq pq$ and $r \in \{p, q\}$, $r \in f$ such that $\eta(f, r) \cap V(Q) \neq \emptyset$.
\end{lemma}

\begin{proof}
    Since $z \notin \eta(pq)$ and $z$ is peripheral in $(H, \eta)$ then $z \notin A$.
    Suppose $V(Q) \cap N[A] \neq \emptyset$ and let $i$ be minimum such that $x_i \in N[A]$.
    
    If $i = 0$, then $z \in N[A]$. Let $f$ be the unique edge of $H$ such that $z \in \eta(f)$. Then, $f \neq pq$ by assumption. Furthermore, since $z \in N[A] \setminus A = N(A)$, there exists $r \in \{p, q\}$ such that $z \in \eta(f, r)$.
    
    Otherwise, by minimality of $i$, we have $x_{i-1} \notin N[A] \Rightarrow x_i \notin A$. Thus, $x_i \in N[A] \setminus A = N(A)$ so there exists an edge $f$ and $r \in \{p, q\}$ such that $x_i \in \eta(f, r)$. Since $x_i \notin A$ then $f \neq pq$, which concludes the proof.
\end{proof}

\begin{lemma} \label{lem:gyarfas-path-inside-edges}
Let $(H, \eta)$ be an extended strip decomposition of a graph $G$ and $x, y \in V(G)$ be two distinct peripheral vertices in $(H, \eta)$. Let $Q = (x = x_0, x_1, \ldots, x_{k-1}, x_k = y)$ be an induced path in $G$ with endpoints $x$ and $y$.
Then, for every edge $e = ab \in E(H)$, it holds that $|V(Q) \cap \eta(e, a)| \leq 1$, and $V(Q) \subseteq \bigcup_{e \in E(H)} \eta(e)$.
\end{lemma}

\begin{proof}
First, we prove that if there exists $v \in V(H)$ and $x_i \in V(Q)$ such that $x_i \in \eta(v)$, then there exist $p < i < q$ and an edge $e$ incident to $v$ such that $x_p, x_q \in \eta(e, v)$.
Indeed, suppose that $x_i \in V(Q) \cap \eta(v)$ for some $v \in V(H)$. Let $p < i$ be maximum such that $x_p \notin \eta(v)$ (this is well-defined since $x_0 = x \notin \eta(v)$ as $x$ is peripheral in $(H, \eta)$). 
Observe that $x_{p+1} \in \eta(v)$. Since $x_px_{p+1} \in E(G)$, there exists an edge $e \in E(H)$ incident to $v$ such that $x_p \in \eta(e, v)$. 
Similarly, there exists $q > i$ and an edge $f \in E(H)$ incident to $v$ such that $x_q \in \eta(f, v)$.
Since $Q$ is an induced path and $q > p + 1$, there is no edge $x_px_q$ in $G$ and thus $e = f$. Thus, $x_p, x_q \in \eta(e, v)$.
With the same argument, it also holds that if there exists $t \in T(H)$ and $x_i \in V(Q)$ such that $x_i \in \eta(t)$ then there exist $p < i < q$, an edge $e \in E(t)$ and a vertex $v \in V(t)$ such that $x_p, x_q \in \eta(e, v)$.
Therefore, it suffices to prove that $|V(Q) \cap \eta(e, a)| \leq 1$ for every edge $e = ab \in E(H)$ to prove that $V(Q) \subseteq \bigcup_{e \in E(H)} \eta(e)$.

For $uv \in E(H)$, let $L(uv, u)$ and $R(uv, u)$ be minimum and maximum values of $i$ such that $x_i \in \eta(uv, u)$, or $\perp$ if no such values exist.  We claim that if $|V(Q) \cap \eta(uv, u)| \ge 2$, then $L(uv, v) \le L(uv, u), R(uv, u) \le R(uv, v)$ (and, in particular, $|V(Q) \cap \eta(uv, v)| \ge 2$) and prove that as follows. Let $L(uv, u) = i < j = R(uv, u)$. If $x_j \in \eta(uv, v)$, then $R(uv, v) \ge j = R(uv, u)$. If $x_j \not\in \eta(uv, v)$, then let us consider the set $A = \eta(u) \cup \eta(uv) \setminus \eta(uv, v)$ and the smallest $j'>j$ such that $x_{j'} \not\in A$. Such $j'$ is well defined as $y \not\in A$ since $|\eta(uv, u)| \ge 2$. If $j' > j+1$, then $x_{j'-1} \in A$ from the minimality of $j'$, whereas if $j'=j+1$, then $x_{j'-1}=x_j \in A$. So, in any case $x_{j'-1} \in A$ and we deduce that $x_{j'} \in N(A) \subseteq \eta(uv, v) \cup \bigcup_{uw \in E(H), w \neq v} \eta(uw, u)$. However, if $x_{j'} \in \eta(uw, u)$ for some $w \neq v$, then $x_ix_{j'} \in E(G)$, which is a contradiction as $Q$ is an induced path and $j'>j>i$. Hence, $x_{j'} \in \eta(uv, v) \Rightarrow R(uv, v) \ge j' > R(uv, u)$. In either case, we have that $R(uv, v) \ge R(uv, u)$. Similarly, we conclude $L(uv, v) \le L(uv, u)$, which proves our auxiliary claim. However, as $|V(Q) \cap \eta(uv, v)| \ge 2$, the roles of $u$ and $v$ can be reversed to show that $L(uv, u) \le L(uv, v)$ and $R(uv, v) \le R(uv, u)$, hence we conclude that $L(uv, u) = L(uv, v)$ and $R(uv, u) = R(uv, v)$.

By contradiction, suppose that there exists an edge $e = ab \in E(H)$ such that $|V(Q) \cap \eta(e, a)| \geq 2$. Let $L(ab, b) = i$ and $R(ab, b) = j$. We have that $i<j$ and $L(ab, a)=i$, $R(ab, a) = j$, hence $x_i, x_j \in \eta(ab, a) \cap \eta(ab, b)$. Let us consider the set $A = \eta(a) \cup \eta(b) \cup \eta(ab) \cup \bigcup_{abc \in T(H)} \eta(abc)$ and the smallest $j'>j$ such that $x_{j'} \not\in A$. Such $j'$ is well defined as $|\eta(ab, a)|, |\eta(ab, b)| \ge 2$, so $y \not\in A$. We have that $x_{j'} \in N(A)$, however $N(A) \subseteq \bigcup_{ac \in E(H), c \neq b} \eta(ac, a) \cup \bigcup_{bc \in E(H), c \neq a} \eta(bc, b) $, therefore $N(A)$ is complete to $\eta(ab, a) \cap \eta(ab, b)$. Hence $x_ix_{j'} \in E(H)$, which is a contradiction since $Q$ is an induced path and $j'>j>i$, which concludes the proof. 
\end{proof}

\begin{corollary} \label{cor:path-inside-edges-internal}
Let $(H, \eta)$ be an extended strip decomposition of a graph $G$ and $x, y \in V(G)$ be two distinct peripheral vertices in $(H, \eta)$. Let $Q = (x = x_0, x_1, \ldots, x_{k-1}, x_k = y)$ be an induced path in $G$ with endpoints $x$ and $y$.
For every $e = ab \in E(H)$ such that $V(Q) \cap \eta(e) \neq \emptyset$, $G[V(Q) \cap \eta(e)]$ is a path with endpoints in $\eta(e, a)$ and $\eta(e, b)$ and all internal vertices in $\eta(e) \setminus (\eta(e, a) \cup \eta(e, b))$.
\end{corollary}

\begin{proof}
Since $x, y \in V(G)$ are peripheral in $(H, \eta)$, there exist edges $e_x = u_xv_x, e_y = u_yv_y \in E(H)$ with $v_x, v_y$ being degree-one vertices, such that $\eta(e_x, v_x) = \{x\}$ and $\eta(e_y, v_y) = \{y\}$.

Let $G_E = G[\bigcup_{e \in E(H)} \eta(e)]$. By \cref{lem:gyarfas-path-inside-edges} we have that $Q \subseteq G_E$.

Let $e = ab \in E(H)$ such that $V(Q) \cap \eta(e) \neq \emptyset$. Let $i$ be minimum such that $x_i \in \eta(e)$ and $j$ be maximum such that $x_j \in \eta(e)$.
If $i = 0$ then $x_i = x$ so $e = e_x$ and up to renaming $a$ and $b$ we have $x_i \in \eta(e, a)$.
If $i > 0$ then $x_{i-1} \notin \eta(e)$ and $x_{i-1}x_i \in E(G)$ so up to renaming $a$ and $b$ we have $x_i \in \eta(e, a)$. Similarly, we have that $x_j \in \eta(e, a) \cup \eta(e, b)$.

Let us first consider the cases where $i=j$.
\begin{itemize}
    \item If $i = j = 0$. Then, $e = e_x = u_xv_x$ and $x_i = x \in \eta(e_x, v_x)$. Furthermore, $x_{i+1} = x_{j+1} \notin \eta(e_x)$ by maximality of $j$, and $v_x$ has degree 1 in $H$ and $\eta(v_x) = \emptyset$ since $x$ is peripheral so $x \in \eta(e, u_x)$.
    \item If $i = j = k$. This is similar to the previous case.

    \item If $0 < i = j < k$. In this case we need to prove that $x_i \in \eta(e, b)$. Let us assume by contradiction that this is false. We have that $x_{i-1}, x_{i+1} \in N(x_i) \cap V(G_E) - \eta(e)$. As $x_i \in \eta(e, a) - \eta(e, b)$, we conclude that $x_{i-1}, x_{i+1} \in \bigcup_{ac \in E(H), c \neq b} \eta(ac, a)$. Let us write then that $x_{i-1} \in \eta(ac_1, a)$ and $x_{i+1} \in \eta(ac_2, a)$. If $c_1=c_2$, then $|V(Q) \cap \eta(ac_1, a)| \ge 2$, which is a contradiction with \cref{lem:gyarfas-path-inside-edges}. If $c_1 \neq c_2$, then $x_{i-1}x_{i+1} \in E(G)$, which is a contradiction with $Q$ being an induced path.
\end{itemize}
From now on, let us assume that $i<j$. Since $|V(Q) \cap \eta(e, a)|, |V(Q) \cap \eta(e, b)| \le 1$ by \cref{lem:gyarfas-path-inside-edges}, we have that $x_i \in \eta(e, a) - \eta(e, b)$, $x_j \in \eta(e, b) - \eta(e, a)$ and no other vertices of $Q$ belong to $\eta(e, a) \cup \eta(e, b)$. Moreover, $\eta(e, a) \cup \eta(e, b)$ separates $\eta(e) - (\eta(e, a) \cup \eta(e, b))$ from $V(G_E) - \eta(e)$, so proving that $x_{i+1} \in \eta(e)$ would imply that all internal vertices of the path $x_i x_{i+1} \ldots x_{j-1} x_j$ belong to $\eta(e) - (\eta(e, a) \cup \eta(e, b))$, as desired, so this is what we are going to prove in the remaining cases.

\begin{itemize}
    \item If $i = 0 < j$. Then $e = e_x = u_xv_x$ and $x_i = x \in \eta(e_x, v_x)$, so we have that $a=v_x, b=u_x$. We recall $x_i \in \eta(e, a) - \eta(e, b)$ and $x_j \in \eta(e, b) - \eta(e, a)$.
    As $a$ is a degree one vertex and $x_0 \notin \eta(e, b)$, we clearly have that $x_1 \in \eta(e)$, which proves our claim.
    
    \item If $i < j = k$. This is similar to the previous case.
    
    \item If $0 < i < j < k$. We recall that $x_i \in \eta(e, a) \setminus \eta(e, b)$ and $x_j \in \eta(e, b) \setminus \eta(e, a)$. 
    As $x_i \in \eta(e, a) - \eta(e, b)$, we can use the same reasoning as in the $0<i=j<k$ case to argue that $x_{i+1} \in \eta(e)$ (which operated under the same assumption that $0<i$ and $x_i \in \eta(e, a) - \eta(e, b)$, but excluded the possibility that $x_{i+1} \in \eta(e)$. That again completes the argument.
    \qedhere
\end{itemize}
\end{proof}

We can now prove \cref{thm:main-thm}.
\begin{proof}
    Let $Q$ be a minimal induced path in $G$ given by~\cref{thm:gyarfas}, that is, a path such that every connected component of $G - N[V(Q)]$ has weight at most  $\wei(V)/2$ and $Q$ is minimal with this property.
    If $|V(Q)| \leq 3t+11$, set $\mathcal{S} = V(Q)$ and observe that every connected component of $G - N[\mathcal{S}]$ has weight at most $\wei(V)/2$. Thus, $G - N[\mathcal{S}]$ has a trivial rigid extended strip decomposition whose every particle has weight at most $\wei(V)/2$.
    Assume now that $|V(Q)| > 3t+11$.
    Let $\ell$ be the last vertex of $Q$.
    Let $Q_2$ be the subpath of $Q$ induced by the last $t+2$ vertices of $Q$, and denote by $z$ the first vertex of $Q_2$.
    Let $z'$ be the predecessor of $z$ in $Q$, and $y$ the predecessor of $z'$ in $Q$. Let $x$ be the first vertex of $Q$, and $Q_1$ be the subpath of $Q$ from $x$ to $y$. Observe that $|V(Q_1)| > 2t + 8$.
      Let $Q_x$, $Q_y$, and $Q_z$ be the sets containing the first $t+1$ successors of $x$, predecessors of $y$, and successors of $z$ in $Q$, respectively (including $x$, $y$, $z$) in $Q$.
    \begin{figure}
        \centering
        \includegraphics[width=1\textwidth]{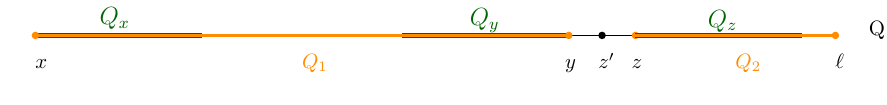}
        \caption{Notations from the proof of \cref{thm:main-thm}. $|Q_x|=|Q_y|=|Q_z|=t+1$. }
        \label{fig:line-segments}
    \end{figure}

    Let $Y = Q_x \cup Q_y \cup Q_z$. Then, $|Y| = 3t+3$.
    Consider the graph $G' = G - (N(Y) - V(Q_1) - V(Q_2))$.
    Note that every vertex of $Q$ is a vertex of $G'$, except $z'$.
    Let $Z = \{x, y, z\}$.
    By \cref{thm:three-in-a-tree}, we can find in polynomial time either an induced subtree of $G'$ containing all three elements of $Z$, or a rigid extended strip decomposition of $(G', Z)$.
    However, any induced subtree of $G'$ containing all three elements of $Z$ contains an induced $S_{t, t, t}$ since it must contain the three pairwise non-touching induced paths $G[Q_x], G[Q_y], G[Q_z]$, each containing $t$ edges.
    Therefore, we can assume that we find a rigid extended strip decomposition $(H, \eta)$ of $(G', Z)$.
    Note that the size of $H$ is polynomial in the size of $G$.
    Suppose first that every particle of $(H, \eta)$ has weight at most $\wei(V)/2$. Applying \cref{lem:make-rigid} to $G - N_G[Y]$ and the restriction of $(H, \eta)$, with $w = \wei(V)/2$, we can find in polynomial time either a rigid extended strip decomposition of $G - N_G[Y]$ whose every particle has weight at most $\wei(V)/2$, in which case we are done setting $\mathcal{S} = Y$, or a set $X \subseteq V(G - N_G[Y])$ of size at most $2$ such that every connected component of $(G-N_G[Y]) - N_{G - N_G[Y]}[X] = G-N_G[Y \cup X]$ has weight at most $\wei(V)/2$, in which case we are done setting $\mathcal{S} = Y \cup X$. 
    
    Suppose now that there exists a particle $A$ such that $\wei(A) > \wei(V)/2$.
    Without loss of generality, we can assume that whenever $p \in V(H)$ is isolated then $G[\eta(p)]$ is connected.
    Consider such a $p \in V(H)$.
    Since $Q$ is a path between peripheral vertices then $N_G[V(Q)] \cap \eta(p) = \emptyset$ so $G[\eta(p)]$ is a connected component of $G - N_G[V(Q)]$, which implies $\wei(\eta(p)) \leq \wei(V)/2$.
    Since every nontrivial particle is contained in a full edge particle we can assume without loss of generality that $A$ is a full edge particle: $A = A_{pq}^{pq}$ for some edge $pq \in E(H)$.
    Applying \cref{lem:gyarfas-path-inside-edges} to the induced path $Q_1$ in $G'$, we obtain that $V(Q_1) \subseteq \bigcup_{e \in E(H)} \eta(e)$.

    \begin{claim}\label{cl:1}
        $V(Q_1) \cap \eta(pq) = \emptyset$.
    \end{claim}
    \begin{proof}
    By contradiction suppose that $V(Q_1) \cap \eta(pq) \neq \emptyset$.
    
    By \cref{cor:path-inside-edges-internal}, $G'[V(Q_1) \cap \eta(pq)]$ is a path with endpoints $v_p \in \eta(pq, p)$ and $v_q \in \eta(pq, q)$.
    Therefore, $\eta(pq, p) \neq \{z\}$ and $\eta(pq, q) \neq \{z\}$ so $z \notin \eta(pq)$.
    Applying \cref{lem:enter-particle} to the induced path $Q_2$ with endpoint $z$ in $G'$, we get that if $V(Q_2) \cap N_{G'}[A] \neq \emptyset$, there exists an edge $f \neq pq$, $r \in \{p, q\}$ and $z'' \in \eta(f, r) \cap V(Q_2)$.
    However, we would then have $z''v_r \in E(G)$ with $z'' \in V(Q_2)$ and $v_r \in V(Q_1)$, contradicting that $Q$ is an induced path in $G$.
    Therefore, $V(Q_2) \cap N_{G'}[A] = \emptyset$.
    
    By minimality of $Q$, $G - N_G[V(Q) - \{\ell\}]$ has a connected component $C$ such that $\wei(V(C)) > \wei(V)/2$.
    Therefore, $\ell$ has a neighbor in $V(C)$, and $V(C) \cap A \neq \emptyset$ since $\wei(V(C)), \wei(A) > \wei(V)/2$.
    Observe also that $C \subseteq V(G) - N_G[V(Q) \setminus \{\ell\}] \subseteq V(G')$
    If $u \in V(C) \cap N_{G'}(A)$ then there exists $f \neq pq \in E(H)$ and $r \in \{p, q\}$ such that $u \in \eta(f, r)$. However, this implies that $uv_r \in E(G') \subseteq E(G)$ and therefore $u \in N_G[V(Q_1)]$, which is impossible since $V(C) \cap N_G[V(Q_1)] = \emptyset$.
    Thus, $V(C) \cap N_{G'}(A) = \emptyset$, and therefore $V(C) \cap A \neq \emptyset$ implies $V(C) \subseteq A$ since $C$ is connected in $G'$.
    Hence, $\ell$ has a neighbor in $A$, which contradicts $V(Q_2) \cap N_{G'}[A] = \emptyset$.

    This finishes the proof of the claim.
    \end{proof}
    
    From Claim~\ref{cl:1} and~\cref{lem:gyarfas-path-inside-edges} we infer that $V(Q_1) \cap A = \emptyset$ as $x$ and $y$ are peripheral.
    Let $X = N_{G'}[A] \cap V(Q_1)$.
    If $u \in X$, there is a vertex $a \in A$ such that $ua \in E(G)$, and thus there exists $r \in \{p, q\}$ and an edge $f$ incident to $r$ such that $u \in \eta(f, r)$.
    However, by \cref{lem:gyarfas-path-inside-edges}, we have $|V(Q_1) \cap \eta(ab, a)| \leq 1$ for every edge $ab \in E(H)$.
    Furthermore, if $u_1 \in \eta(f_1, r), u_2 \in \eta(f_2, r), u_3 \in \eta(f_3, r)$ for some pairwise distinct edges $f_1, f_2, f_3$ then $G[\{u_1, u_2, u_3\}]$ induces a triangle.
    Since $Q_1$ is an induced path, it follows that $|X| \leq 4$.
    By \cref{lem:small-dom-ngbhd}, there exists a set $X_A$ of size at most 2 such that $N_{G'}(A) \subseteq N_{G'}(X_A)$.
    \begin{claim}\label{cl:2}
    Every connected component of $G'' := G - N_G[Y \cup X \cup X_A \cup \{\ell, z'\}]$ has weight at most $\wei(V)/2$.
    \end{claim}
    \begin{proof}
    Consider a connected component $C$ of $G''$ and assume by contradiction that $\wei(V(C)) > \wei(V)/2$. Then, $V(C) \cap A \neq \emptyset$.
    If $v \in V(G'') \cap N_G[A]$ then $v \in V(G')$ and $v \notin N_G[X_A] \supseteq N_{G'}(X_A) \supseteq N_{G'}(A)$. Therefore, $v \in N_{G'}[A] \setminus N_{G'}(A) = A$.
    Since $V(C) \cap A \neq \emptyset$ and $C$ is connected in $G''$ then $V(C) \subseteq A$.
    However, every connected component of $G - N_G[V(Q)]$ has weight at most $\wei(V)/2$. 
    Thus, $C$ is not contained in any connected component of $G - N_G[V(Q)]$ and therefore there exists some vertex $q \in V(Q) \setminus N_G[Y \cup X \cup X_A \cup \{\ell, z'\}]$ which is in $N_G[V(C)]$.
    Therefore, $q \in V(Q_1) \cap N_G[V(C)] \subseteq V(Q_1) \cap N_G[A] = X$, a contradiction.
    This finishes the proof of the claim.
    \end{proof}
    Thanks to Claim~\ref{cl:2}, 
    by setting $\mathcal{S} = Y \cup X \cup X_A \cup \{\ell, z'\}$, we have $|\mathcal{S}| \leq 3t+11$ and $G - N_G[\mathcal{S}]$ has a trivial rigid extended strip decomposition whose every particle has weight at most $\wei(V)/2$. 
    
    Finally, note that this proof is constructive and immediately yields a polynomial time algorithm.
\end{proof}

\bibliography{main}
\end{document}